\documentclass[twoside,leqno,10pt, A4]{amsart}
\usepackage{amsfonts}
\usepackage{amsmath}
\usepackage{amscd}
\usepackage{amssymb}
\usepackage{amsthm}
\usepackage{amsrefs}
\usepackage{latexsym}
\usepackage{mathrsfs}
\usepackage{bbm}
\usepackage{color}
\begin{document}

\newtheorem{theorem}[subsection]{Theorem}
\newtheorem{proposition}[subsection]{Proposition}
\newtheorem{lemma}[subsection]{Lemma}
\newtheorem{corollary}[subsection]{Corollary}
\newtheorem{conjecture}[subsection]{Conjecture}
\newtheorem{prop}[subsection]{Proposition}
\numberwithin{equation}{section}
\newcommand{\mr}{\ensuremath{\mathbb R}}
\newcommand{\mc}{\ensuremath{\mathbb C}}
\newcommand{\dif}{\mathrm{d}}
\newcommand{\intz}{\mathbb{Z}}
\newcommand{\ratq}{\mathbb{Q}}
\newcommand{\natn}{\mathbb{N}}
\newcommand{\comc}{\mathbb{C}}
\newcommand{\rear}{\mathbb{R}}
\newcommand{\prip}{\mathbb{P}}
\newcommand{\uph}{\mathbb{H}}
\newcommand{\fief}{\mathbb{F}}
\newcommand{\majorarc}{\mathfrak{M}}
\newcommand{\minorarc}{\mathfrak{m}}
\newcommand{\sings}{\mathfrak{S}}
\newcommand{\fA}{\ensuremath{\mathfrak A}}
\newcommand{\mn}{\ensuremath{\mathbb N}}
\newcommand{\mq}{\ensuremath{\mathbb Q}}
\newcommand{\half}{\tfrac{1}{2}}
\newcommand{\f}{f\times \chi}
\newcommand{\summ}{\mathop{{\sum}^{\star}}}
\newcommand{\chiq}{\chi \bmod q}
\newcommand{\chidb}{\chi \bmod db}
\newcommand{\chid}{\chi \bmod d}
\newcommand{\sym}{\text{sym}^2}
\newcommand{\hhalf}{\tfrac{1}{2}}
\newcommand{\sumstar}{\sideset{}{^*}\sum}
\newcommand{\sumprime}{\sideset{}{'}\sum}
\newcommand{\sumprimeprime}{\sideset{}{''}\sum}
\newcommand{\shortmod}{\ensuremath{\negthickspace \negthickspace \negthickspace \pmod}}
\newcommand{\V}{V\left(\frac{nm}{q^2}\right)}
\newcommand{\sumi}{\mathop{{\sum}^{\dagger}}}
\newcommand{\mz}{\ensuremath{\mathbb Z}}
\newcommand{\leg}[2]{\left(\frac{#1}{#2}\right)}
\newcommand{\muK}{\mu_{\omega}}

\title[First Moments of Some {H}ecke {$L$}-functions of Prime Modulus]{First Moments of Some {H}ecke {$L$}-functions of Prime Moduli}

\date{\today}
\author{Peng Gao and Liangyi Zhao}

\begin{abstract}
 We study the first moments of central values of Hecke $L$-functions associated with quadratic, cubic and quartic symbols to prime moduli. This also enables us
to obtain results on first moments of central values of certain families of cubic and quartic Dirichlet $L$-functions of prime moduli.

\end{abstract}

\maketitle



\section{Introduction}
\label{sec 1}

Moments of $L$-functions can be used to address the non-vanishing of the central values of $L$-functions,
an issue that attracts much attention in the literature due to significant arithmetic information carried by these central values.  For example, the central values of quadratic Dirichlet $L$-functions are related to class numbers of imaginary quadratic fields (\cite[p. 514]{iwakow}), so that a considerable amount of investigations have been done on moments of this family of $L$-functions. \newline

 In \cite{Jutila}, M. Jutila initiated the study on the first and second moments central values of the family of quadratic $L$-functions and
used his result to show that there are infinitely many $L$-functions in this family with non-vanishing central values.
The error terms in Jutila's result was subsequently improved in \cites{DoHo, MPY, ViTa} for the first moment and \cites{sound1, Sono} for the second moment.
A valid asymptotic formula for the third moment of central values of the family of quadratic Dirichlet $L$-functions was first established by K. Soundararajan
in \cite{sound1} and the error term for a smoothed version was later improved by M. P. Young in \cite{Young2}. Recently, Q. Shen obtained
a valid asymptotic formula in \cite{Shen} for the fourth moment of central values of the same family of $L$-functions under the Generalized Riemann Hypothesis (GRH). \newline

  Due to its relation to the Birch and Swinnerton-Dyer conjecture, the family of quadratic twists of a modular form is another important family.
The mean value of that family has been studied in \cites{BFH, Iwan1, Munshi, MM, Petrow1}. Assuming GRH, the second moment of the family was computed by  K. Soundararajan and M. P. Young in \cite{S&Y}. \newline

In addition to the above mentioned families,  much work has been done on moments of Hecke $L$-functions associated with various families of characters of a
fixed order. In \cite{Luo}, W. Luo studied the first two moments of cubic Hecke $L$-functions in $\mq(\omega)$, where $\omega=\frac {-1+\sqrt{3}i}{2}$ is a
primitive cubic root of unity.  The analogue case for cubic Dirichlet $L$-functions was obtained by S. Baier and M. P. Young in \cite{B&Y}. We refer the reader to
the articles \cites{FaHL, GHP, FHL, Diac, G&Zhao1} for related results in this area. \newline

  All the above mentioned results have the common feature that the set of conductors of each family of $L$-functions considered has a positive density in the ring of integers in their respective number fields. Keeping in mind that the set of rational primes is a sparse set (a subset of density zero in the set of rational integers), it is then interesting to
consider the case when the set of conductors of a family forms a sparse set.
In fact, a weighted first moment of central values of the family of quadratic Dirichlet $L$-functions of prime conductors has already been studied by Jutila
in \cite{Jutila}, the same paper in which he obtained the first two moments of the same family of $L$-functions when the set of conductors of the family
has a positive density. A weighted second moment of
central values of the family of quadratic Dirichlet $L$-functions of prime conductors was recently computed by S. Baluyot and K. Pratt in \cite{BP} assuming GRH. \newline

  Motivated by the above results of Jutila, we study, in this paper, the first moments of a few families of $L$-functions with prime moduli assuming GRH. Our first result is an analogue of Jutila's result on quadratic Hecke $L$-functions.
Let $K$ be a number field of class number one, $\mathcal{O}_K$ be its ring of integers and $U_K$ the group of units in $\mathcal{O}_K$.
We shall denote $\varpi$ for a prime number in $\mathcal{O}_K$, by which we mean that the ideal $(\varpi)$ generated by $\varpi$ is a prime ideal.
We write $N(k), \mathrm{Tr}(k)$ for the norm and trace of any $k \in K$.
We further denote $\chi$ for a Hecke character of $K$ and we say that $\chi$ is of trivial infinite type if its component at infinite places of $K$ is trivial.
We write $L(s,\chi)$ for the $L$-function associated to $\chi$ and $\zeta_{K}(s)$ for the Dedekind zeta function of $K$ and $\zeta(s)$
for the Riemann zeta function.  We also use $\Lambda(n)$ for the von Mangoldt function on $\mathcal{O}_K$ given by
\begin{align*}
    \Lambda(n)=\begin{cases}
   \log N(\varpi) \qquad & n=\varpi^k, \text{$\varpi$ prime}, k \geq 1, \\
     0 \qquad & \text{otherwise}.
    \end{cases}
\end{align*}

   In the remainder of this section, we let $K=\mq(\omega)$ or $\mq(i)$. Let $\chi_{\varpi}$ be a quadratic Hecke symbol defined in Section \ref{sec2.4} and it is shown there
that $\chi_{\varpi}$ is a Hecke character of trivial infinite type when $\varpi \equiv 1 \pmod {36}, \varpi \in \mz[\omega]$ or
$\varpi \equiv 1 \pmod {16}, \varpi \in \mz[i]$. For the corresponding families of Hecke $L$-functions, we have the following
\begin{theorem}
\label{firstmoment}
  Suppose that GRH is true. Let $\Phi$ be a compactly supported smooth Schwartz class function. For $y \rightarrow \infty$ and any $\varepsilon > 0$, we have
\begin{align} \label{1stmom}
    \sumstar_{\varpi}L \left( \frac{1}{2},
   \chi_{\varpi} \right) \Lambda(\varpi) \Phi\left( \frac{N(\varpi)}{y} \right) =A_K \hat{\Phi}(0) y\log y +B_K \hat{\Phi}(0)y+O\left( y^{3/4+\varepsilon} \right),
\end{align}
   where $K = \ratq(i)$ or $\ratq(\omega)$, $\sum^{*}$ indicates that the sum runs over prime elements of $\mz[\omega]$ congruent to $1 \pmod{36}$
if $K = \ratq(\omega)$,  or prime elements of $\mz[i]$ congruent to $1 \pmod {16}$ if $K = \ratq(i)$, $\chi_{\varpi}$ is the corresponding quadratic symbol, $B_K$ is a constant depending on $K$ and $\Phi$, and
\begin{align} \label{1.4}
   A_{\ratq(\omega)}= \frac {(1+\sqrt{3})\pi}{1296},
\quad A_{\ratq(i)}= \frac {(2+\sqrt{2})\pi}{512}, \quad \hat{\Phi}(0) = \int\limits_1^2 \Phi(x) \dif x.
\end{align}
\end{theorem}

  The proof of Theorem \ref{firstmoment} is given in Section \ref{section Thm1}. As we are summing over primes, we cannot apply the Poisson summation (due to K. Soundararajan in \cite{sound1}) to deal with these sums. Thus our treatment here is similar to that of Jutila in \cite{Jutila},
except that we use the assumption on GRH to get better error terms. \newline

  Let $\chi_{j, \varpi}, j=3$, $4$ be a cubic or quartic Hecke symbol defined in Section \ref{sec2.4}.  It is shown there
that $\chi_{3, \varpi}$ is a Hecke character of trivial infinite type when $\varpi \equiv 1 \pmod {9}, \varpi \in \mz[\omega]$ and
$\chi_{4, \varpi}$ is a Hecke character of trivial infinite type when $\varpi \equiv 1 \pmod {16}, \varpi \in \mz[i]$.  Our next result deals with the associated cubic or quartic Hecke $L$-functions.
\begin{theorem}
\label{secmom}
 Suppose that GRH is true. Let $\Phi$ be a compactly supported smooth Schwartz class function. For $y \rightarrow \infty$ and any $\varepsilon > 0$, we have
\begin{align*}
    \sumstar_{\varpi}L \left( \frac{1}{2},
   \chi_{j,\varpi} \right) \Lambda(\varpi) \Phi\left( \frac{N(\varpi)}{y} \right) =C_K \hat{\Phi}(0)y+O\left( y^{39/40+\varepsilon} \right),
\end{align*}
   where $K = \ratq(i)$ or $\ratq(\omega)$,  $\hat{\Phi}(0)$ is given in \eqref{1.4},  $\sum^{*}$ indicates that the sum runs over prime elements of $\mz[\omega]$ congruent to $1 \pmod{9}$
if $K = \ratq(\omega)$,  or prime elements of $\mz[i]$ congruent to $1 \pmod {16}$ if $K = \ratq(i)$, $\chi_{j, \varpi}$ is the corresponding cubic symbol or quartic symbol with $j=3$ when $K=\mq(\omega)$
and $j=4$ when $K=\mq(i)$, and
\begin{align} \label{C}
   C_{\ratq(\omega)}= \frac {3\sqrt{3}-1}{27(\sqrt{3}-1)}  \zeta_{\mq(\omega)} \left( \frac 32 \right),
\quad C_{\ratq(i)}= \frac {3(2+\sqrt{2})}{128}\zeta_{\mq(i)}(4).
\end{align}
\end{theorem}

   The proof of Theorem \ref{secmom} is given in Section \ref{section Thm2} and our approach is analogous to that used by W. Luo in \cite{Luo}.
In particular, we make crucial use (see Lemma \ref{lemg3} below) of
a result of S. J. Patterson in \cite{P} on estimations of certain Gauss sums over primes to control the error term. \newline

   Let $\chi$ be a Dirichlet character modulo $q$ and $\chi_0$ the principal character modulo $q$ and we shall
say that the order of $\chi$ is $j\in \natn$ if $\chi^j=\chi_0$ but $\chi^i \neq \chi_0$ for all $1 \leq i<j$. We denote the order of $\chi$ by $\text{ord}(\chi)$.
By Lemma \ref{lemma:quarticclass} below, we see that each cubic or quartic symbol $\chi_{j, \varpi}$, $j=3$, $4$ gives rise to a corresponding Dirichlet character of
order $3$ or $4$ modulo $N(\varpi)$. One can consider the first moments of these $L$-functions and our result is
\begin{theorem}
\label{thirdmom}
 Suppose that GRH is true. Let $j=3,4$ and $\Phi$ be a compactly supported smooth Schwartz class function. For $Q \rightarrow \infty$ and any $\varepsilon > 0$, we have
\begin{align*}
     \sumstar_{\substack{p} }\;
\sum_{\substack{\chi \bmod{p} \\ \text{ord}(\chi)=j }} L \left( \frac 1{2}, \chi \right) \Lambda(p) \Phi \leg{p}{Q} =&
D_j \hat{\Phi}(0)Q+O\left( Q^{39/40+\varepsilon} \right),
\end{align*}
   where $\hat{\Phi}(0)$ is given in \eqref{1.4}, $\sum^{*}$ indicates that the sum runs over prime numbers $p$ of $\mz$ such that
$p=N(\varpi)$ with $\varpi \equiv 1 \bmod {9}, \varpi \in \mz[\omega]$ when $j=3$
or $\varpi \equiv 1 \bmod {16}, \varpi \in \mz[i]$ when $j=4$, and
\begin{align} \label{D}
   D_{3}= \frac {3\sqrt{3}-1}{27(\sqrt{3}-1)}\zeta \left( \frac 32 \right),
\quad D_{4}= \frac {3(2+\sqrt{2})}{128}\zeta(2).
\end{align}
\end{theorem}

   The proof of Theorem \ref{thirdmom} is similar to that of Theorem \ref{secmom} and will be given in Section \ref{sec thirdthm}.
\section{Preliminaries}
\label{sec 2}

In this section, we include some auxiliary results needed in the proofs of our theorems.

\subsection{Residue symbols and Gauss sums}
\label{sec2.4}
   Recall from Section \ref{sec 1} that we set $K$ be a number field of class number one.
Let $n \in \natn$ with $n \geq 2$ and $\mu_n(K)=\{ \zeta \in K^{\times}: \zeta^n=1 \}$ and suppose that $\mu_n(K)$ has $n$ elements. Note that the discriminant of $x^n-1$ is divisible only by the primes dividing $n$ in $\mathcal{O}_K$. It follows that for any prime $\varpi \in \mathcal{O}_K,  (\varpi, n)=1$, we have a bijective map
\begin{align*}
    \zeta \mapsto \zeta \pmod \varpi : \mu_n(K) \rightarrow \mu_n(\mathcal{O}_K/\varpi)=\{ \zeta \in (\mathcal{O}_K/\varpi)^{\times}: \zeta^n=1 \}.
\end{align*}
  For such $\varpi$,  we define the $n$-th power residue symbol $\leg{\cdot}{\varpi}_{n,K}$ in $K$ such that $\leg{a}{\varpi}_{n, K} \equiv
a^{(N(\varpi)-1)/n} \pmod{\varpi}$ with $\leg{a}{\varpi}_{n,K} \in \mu_n(K)$ for any $a \in \mathcal{O}_K$, $(a, \varpi)=1$. When
$\varpi | a$, we define $\leg{a}{\varpi}_{n, K} =0$.  Then these symbols can be extended to any composite $c$ with $(N_K(c), n)=1$ multiplicatively. We further define $\leg {\cdot}{c}_{n, K}=1$ when $c \in U_K$. \newline

   In the remainder of this section, we shall let $K=\mq(\omega)$ with $\omega=\exp(2\pi i/3)$ or $K=\mq(i)$ unless otherwise specified.
It is well-known that both fields have class number one and $\mathcal{O}_{K}=\mz[\omega]$, $\mz[i]$, respectively.
We use $\delta_K$ and $D_K$ for the different and discriminant of $K$, respectively.
In particular, we fix $\delta_{\mq(\omega)}=\sqrt{-3}$, $\delta_{\mq(i)}=2i$, $D_{\mq(\omega)}=-3$, $D_{\mq(i)}=-4$.
We shall reserve the symbol $\leg {\cdot}{\cdot}_3$ for the cubic residue symbol $\leg {\cdot}{\cdot}_{3, \mq(\omega)}$ and
$\leg {\cdot}{\cdot}_4$ for the quartic residue symbol $\leg {\cdot}{\cdot}_{4, \mq(i)}$ in this paper and
we define $\chi_{3,a}=\leg {\cdot}{a}_{3}$ for any $a \in \mz[\omega]$, $\chi_{4,a}=\leg {\cdot}{a}_{4}$ for any $a \in \mz[i]$.  We shall also write
 $\chi_{a}=\leg {\cdot}{a}_{2,K}$ for any $a \in \mathcal{O}_K$ when there is no confusion about $K$ from the context. \newline

Recall that every ideal in $\intz[\omega]$ co-prime to $3$ has a unique generator congruent to $1$ modulo $3$ (see \cite[Proposition 8.1.4]{BEW})
and every ideal in $\intz[i]$ co-prime to $2$ has a unique generator congruent to $1$ modulo $(1+i)^3$
(see the paragraph above Lemma 8.2.1 in \cite{BEW})). These generators are called primary. An element $n=a+b\omega$ in $\mz[\omega]$ is congruent to $1 \pmod{3}$ if and only if $a \equiv 1 \pmod{3}$, and $b \equiv
0 \pmod{3}$ (see the discussions before \cite[Proposition 9.3.5]{I&R}). \newline

  We refer the reader to \cite[Section 2.2]{G&Zhao2019-1} for the quadratic, cubic and quartic reciprocity laws as well as the supplementary laws are applicable to the above quadratic, cubic and quartic symbols. We only point out here that the quadratic symbol $\leg {\cdot}{c}_{2, \mq(\omega)}$ is trivial on
units for any $c \equiv 1 \pmod {36}, c \in \mz[\omega]$ so that for any square-free  $c$,
$\chi_c$ can be regarded as a primitive character of the ray class group $h_{(c)}$. Here we recall that for any number field $K$ of class number one and any $c \in K$, the ray
class group $h_{(c)}$ is defined to be $I_{(c)}/P_{(c)}$, where
$I_{(c)} = \{ \mathcal{A} \in I : (\mathcal{A}, (c)) = 1 \}$ and
$P_{(c)} = \{(a) \in P : a \equiv 1 \pmod{c} \}$ with $I$ and $P$
denoting the group of fractional ideals in $K$ and the subgroup of
principal ideals, respectively. Moreover, the cubic symbol $\leg {\cdot}{c}_{3}$ is trivial on
units for any $c \equiv 1 \pmod {9}, c \in \mz[\omega]$ so that for any square-free  $c$, $\chi_{3,c}$ can be regarded as a primitive character of the ray class group $h_{(c)}$. We also note that the supplement laws to the cubic reciprocity law \cite[Theorem 7.12]{Lemmermeyer} imply that
\begin{align}
\label{cubicsupplyment}
\leg {2}{c}_{2, \mq(\omega)}=\leg {1-\omega}{c}_{2, \mq(\omega)}=1, \quad c \equiv 1 \pmod {36}, \quad
 \leg {1-\omega}{c}_{3}=1, \quad c \equiv 1 \pmod {9}.
\end{align}

 Similarly, $\leg {\cdot}{c}_{2, \mq(i)}$ and $\leg {\cdot}{c}_{4}$ are trivial on units and $1+i$ for any $c \equiv 1 \pmod {16}, c \in \mz[i]$,
so that for any square-free  $c$, $\chi_c$ and  $\chi_{4, c}$ can be regarded as a primitive character of the ray class group $h_{(c)}$. \newline

  We note that if $\varpi$ is a prime such that $N(\varpi)$ is a rational prime, then restricting $\chi_{3, \varpi}$ or $\chi_{4,\varpi}$
on rational integers gives rise to cubic or quartic Dirichlet characters, we shall say that these Dirichlet characters are induced by $\chi_{3, \varpi}$ or $\chi_{4,\varpi}$.
We have the following classification of primitive cubic and quartic
Dirichlet characters of prime conductors, which is a special case of the one given in \cite[Lemma 2.2]{G&Zhao6}.
\begin{lemma}
\label{lemma:quarticclass}
 The primitive cubic Dirichlet characters of prime conductor $p$ co-prime to $3$ are induced by $\chi_{3, \varpi}$ for some prime
$\varpi \in \mz[\omega]$ such that $N(\varpi) = p$. The primitive quartic Dirichlet characters of prime conductor $p$ co-prime to $2$ are induced
by $\chi_{4, \varpi}$ for some prime
$\varpi \in \mz[i]$ such that $N(\varpi) = p$.
\end{lemma}

  In particular, the above Lemma implies that cubic Dirichlet characters of prime conductor $p$ exist if and only if $p \equiv 1 \pmod 3$, in which case there are
two such characters induced by $\chi_{3, \varpi}$ or $\chi_{3, \overline{\varpi}}$, where $\varpi$ is a prime in $\mz[\omega]$ such that
$N(\varpi)=p$. Also, quartic Dirichlet characters of prime conductor $p$ exist if and only if $p \equiv 1 \pmod 4$, in which case there are
two such characters induced by $\chi_{4, \varpi}$ or $\chi_{4, \overline{\varpi}}$, where $\varpi$ is a prime in $\mz[i]$ such that
$N(\varpi)=p$. \newline

  Now we set $e(x)=\exp(2 \pi i x)$ and let $\chi$ be a Dirichlet character of modulus $n$. For any $r \in \mz$, we define the Gauss sum
$\tau(r, \chi)$ as follows:
\begin{align*}
  \tau(r, \chi)=\sum_{x \bmod {n}}\chi(x)e(rx).
\end{align*}
   We also define $\tau(\chi) =\tau(1, \chi)$. \newline

   We further define $\widetilde{e}_K(k) =e(\mathrm{Tr}(k/ \delta_K))$ for any $k \in K$ and we write $\widetilde{e}_{\omega}(z)$ for $\widetilde{e}_{\mq(\omega)}(k)$ and $\widetilde{e}_{i}(z)$ for $\widetilde{e}_{\mq(i)}(k)$. For any $n, r \in \mathcal{O}_{K}, (n,2)=1$, we define
\begin{align*}
 g_2(r,n) = \sum_{x \bmod{n}} \leg{x}{n}_{2,K} \widetilde{e}_{K}\leg{rx}{n}.
\end{align*}
  We define $g_3(r,n)$ for $r, n \in \mathcal{O}_{\mq(\omega)}, (n,3)=1$ and $g_4(r,n)$ for $r, n \in \mathcal{O}_{\mq(i)}, (n,2)=1$ similarly, replacing $2$
in the above expression by $3$ or $4$ respectively. We shall write $g_i(n)$ for $g_i(1,n)$ in what follows for $i=2$, $3$ and $4$. \newline

If $\chi$ is induced by $\chi_{3,\varpi}$ or  $\chi_{4,\varpi}$ for a primary $\varpi$, then it is shown in \cite[Section 2.2]{G&Zhao6} that
\begin{align*}
  \tau(r, \chi) & =\begin{cases}
    \displaystyle \overline{\leg {\sqrt{-3}}{\varpi}}_3 g_3(r, \varpi)=\overline{\leg {\omega(1-\omega)}{\varpi}}_3 g_3(r, \varpi) \qquad & \text{if $\chi$ is induced by $\chi_{3,\varpi}$}, \\ \\
    \displaystyle \overline{\leg {-i}{\varpi}}_4 (-1)^{(\Re(\varpi)^2-1)/8}g_4(r, \varpi) \qquad & \text{if $\chi$ is induced by $\chi_{4,\varpi}$}.
    \end{cases}
\end{align*}

  We apply the above relations to the case when $\varpi \equiv 1 \pmod {9}$ or $\varpi \equiv 1 \pmod {16}$, by noting \eqref{cubicsupplyment} and that $\chi_{3, \varpi}$ and $\chi_{4,\varpi}$ are
trivial on the units for these $\varpi$ to arrive at
\begin{align}
\label{taucubic}
  \tau(r, \chi)& =\begin{cases}
     g_3(r, \varpi)  \qquad & \text{if $\varpi \equiv 1 \pmod {9}$}, \\
     g_4(r, \varpi) \qquad & \text{if $\varpi \equiv 1 \pmod {16}$}.
    \end{cases}
\end{align}

\subsection{The approximate functional equation}

   Let $K =\mq(i)$ or $\mq(\omega)$ and $\chi$ be a primitive Hecke character $\pmod {m}$ of trivial infinite type defined on $\mathcal{O}_K$.
As shown by E. Hecke, $L(s, \chi)$ admits
analytic continuation to an entire function and satisfies the
functional equation (\cite[Theorem 3.8]{iwakow})
\begin{align}
\label{1.1}
  \Lambda(s, \chi) = W(\chi)(N(m))^{-1/2}\Lambda(1-s, \overline{\chi}),
\end{align}
   where $|W(\chi)|=(N(m))^{1/2}$ and
\begin{align*}
  \Lambda(s, \chi) = (|D_K|N(m))^{s/2}(2\pi)^{-s}\Gamma(s)L(s, \chi).
\end{align*}

   We recall from \cite[Section 2.4]{G&Zhao2019-1} that for any $x>1$, we have the following approximate functional equation for $L(1/2+it, \chi)$:
\begin{equation} \label{approxfuneq}
\begin{split}
 L \left( \frac{1}{2}+it, \chi \right) = \sum_{0 \neq \mathcal{A} \subset
  \mathcal{O}_K} & \frac{\chi(\mathcal{A})}{N(\mathcal{A})^{1/2+it}}V_t \left(\frac{2\pi  N(\mathcal{A})}{x} \right) \\
  & + \frac{W(\chi)}{N(m)^{1/2}}\left(\frac {(2\pi)^2}{|D_k|N(m)} \right )^{it} \frac {\Gamma (1/2-it)}{\Gamma (1/2+it)}\sum_{0 \neq \mathcal{A} \subset
  \mathcal{O}_K}\frac{\overline{\chi}(\mathcal{A})}{N(\mathcal{A})^{1/2-it}}V_{-t}\left(\frac{2\pi
  N(\mathcal{A})x}{|D_K|N(m)} \right),
     \end{split}
\end{equation}
    where
\begin{align}
\label{2.14}
  V_t \left(x \right)=\frac {1}{2\pi
   i}\int\limits\limits_{(2)}\frac {\Gamma(s+1/2+it)}{\Gamma (1/2+it)} \frac
   {x^{-s}}{s} \ \dif s.
\end{align}

    Similarly, let $\chi$ be a primitive Dirichlet character $\chi$ of conductor $q$ such that $\chi(-1)=1$ and let $A$ and $B$ be positive real numbers such that $AB = q$,
we recall the following approximate functional equation for Dirichlet $L$-functions given in \cite[Theorem 5.3]{iwakow}:
\begin{align}
\label{approxfunc}
L \left( \frac 1{2}, \chi \right) = \sum_{m=1}^{\infty} \frac{\chi(m)}{m^{1/2}} W \left(\frac{m}{A}\right) + \frac {\tau(\chi)}{q^{1/2}}
\sum_{m=1}^{\infty} \frac{\overline{\chi}(m)}{m^{1/2}} W \left(\frac{m}{B}\right), \; \mbox{where} \; W(x) = \frac{1}{2\pi i} \int\limits_{(2)}  \pi^{-s/2} \frac{\Gamma\left(1/4 + s/2 \right)}{\Gamma\left(1/4\right)} x^{-s} \frac{ \dif s}{s}.
\end{align}

      We write $V$ for $V_0$ and note that (see \cite[Lemma 2.1]{sound1}) both $V(x)$ and $W(x)$ are real-valued and smooth on $[0, \infty)$ such that
for the $j$-th derivative of $V(x)$ and $W(x)$,
\begin{equation} \label{2.07}
      V\left (x \right), W(x) = 1+O(x^{1/2-\varepsilon}) \; \mbox{for} \; 0<x<1   \quad \mbox{and} \quad V^{(j)}\left (x \right), W^{(j)}\left (x \right) =O(e^{-x}) \; \mbox{for} \; x >0, j \geq 0.
\end{equation}

  When $\chi$ is a quadratic Hecke character, we have $\chi=\overline{\chi}$ so that by setting $x=1/2$ in \eqref{1.1}, we deduce that
\begin{align*}
  W(\chi)=N(m)^{1/2}.
\end{align*}
   It follows from this and \eqref{approxfuneq} by setting $x=(|D_K|N(m))^{1/2}$ that
\begin{equation} \label{quadapproxfuneqQi}
 L \left( \frac{1}{2}, \chi \right) = 2\sum_{0 \neq \mathcal{A} \subset
  \mathcal{O}_K}\frac{\chi_c(\mathcal{A})}{N(\mathcal{A})^{1/2}}V \left(\frac{2\pi N(\mathcal{A}) }{(|D_K|N(m))^{1/2}} \right).
  \end{equation}

\subsection{Estimation of certain Gauss sums}
\label{section: smooth Gauss}

     In the proof of Theorem \ref{secmom} and \ref{thirdmom}, we need a result of S. J. Patterson in \cite{P} to estimate certain Gauss sums over primes. We state
the result in the following
\begin{lemma}
\label{lemg3} Let $\omega$ denote a prime in $\mz[\omega]$ or $\mz[i]$. Let $\psi$ be any ray class character $\pmod 9$ in $\mz[\omega]$ or
any ray class character $\pmod {16}$ in $\mz[i]$.
For $x>1$  we have for any $(a, 3)=1$, $a \in \mz[\omega]$ and any $(b, 2)=1$, $b \in \mz[i]$,
\begin{align*}
\begin{split}
   \sum_{\substack {\varpi \equiv 1 \pmod {9} \\ N(\varpi) \leq x}}\overline{\leg {a}{\varpi}}_3 \frac {\psi(\varpi) g_3(\varpi) \Lambda(\varpi)}{\sqrt{N(\varpi)}}  \ll & x^{\varepsilon}(N(a)^{1/8}x^{27/32}+x^{1-1/20}), \\
  \sum_{\substack {\varpi \equiv 1 \pmod {(1+i)^3} \\ N(\varpi) \leq x}}\overline{\leg {b}{\varpi}}_4 \frac {\psi(\varpi) g_4(\varpi) \Lambda(\varpi)}{\sqrt{N(\varpi)}} \ll & x^{\varepsilon}(N(b)^{1/10}x^{1-1/10}+N(b)^{1/8+\varepsilon}x^{1-1/8}+x^{1-1/20}).
\end{split}
\end{align*}
\end{lemma}
\begin{proof}
  The second estimation follows from the estimation of $E(x;k,l)$ given in the end of Section 4.1 of \cite{G&Zhao4} by setting $k=b$, $l=1$ there and by noting that a further twist
of each term in the sum in $E(x;k,l)$  by $\psi$ will not affect the bound. The first
estimation can be obtained similarly by using the  (with the notations being those used in \cite{G&Zhao4}) bounds
\begin{align*}
  \psi & \ll N(d)^{\frac 1{2} (\frac 32-\Re(s)+\varepsilon)}N(b)^{\frac {3}{4}n-1-\frac {1}{2} n \Re(s)+2\varepsilon}(1+|s|^2)^{\text{Card}\sum_{\infty}(k) \cdot \frac 1{2} (n-1)(\frac 32-\Re(s)+\varepsilon)}, \\
  \text{Res}_{s=1+1/3}\psi & \ll N(d)^{-\frac 1{6}}N(b)^{-1}
\end{align*}
  in the proof of \cite[Lemma 4.2]{G&Zhao4} (the second estimation above being a consequence of \cite[Lemma, p. 200]{P}), and following the arguments there.
\end{proof}

\section{Proof of Theorem~\ref{firstmoment}}
\label{section Thm1}

    As the proofs for $K=\mq(\omega)$ and $\mq(i)$ are similar,  we will only give the proof for the case $K=\mq(\omega)$ here. We shall hence fix $K=\mq(\omega)$
throughout the proof. We apply \eqref{quadapproxfuneqQi} and arrive at
\begin{align}
\label{sumprimetointeger}
\begin{split}
   \sum_{\varpi \equiv 1 \bmod {36}}L \left( \frac{1}{2},
   \chi_{\varpi} \right) \Lambda(\varpi) \Phi\left( \frac{N(\varpi)}{y} \right) =& 2\sum_{\varpi \equiv 1 \bmod {36}} \ \sum_{0 \neq \mathcal{A} \subset
  O_K} \frac{\chi_{\varpi}(\mathcal{A})\Lambda(\varpi) }{N(\mathcal{A})^{1/2}}V \left(\frac{2\pi N(\mathcal{A}) }{(3N(\varpi))^{1/2}} \right)\Phi\left( \frac{N(\varpi)}{y} \right) \\
=& M +O \left( y^{3/4+\varepsilon} \right) ,
\end{split}
\end{align}
  where
\begin{align}
\label{M}
  M=2\sum_{c \equiv 1 \bmod {36}} \ \sum_{0 \neq \mathcal{A} \subset
  O_K} \frac{\chi_{c}(\mathcal{A})\Lambda(c) }{N(\mathcal{A})^{1/2}}V \left(\frac{2\pi N(\mathcal{A}) }{(3N(c))^{1/2}} \right)\Phi\left( \frac{N(c)}{y} \right).
\end{align}
Here the first summation in the defintion of $M$ runs over all elements $c \in \mz[\omega]$ and the last equality in \eqref{sumprimetointeger} follows from the rapid decay of $\Phi$ and $V$ given in \eqref{2.07} so that the contribution from the higher prime power is
\begin{align*}
 \ll \sum_{\substack{ N(\varpi^k) \ll y^{1+\varepsilon} \\ k \geq 2 }}  \sum_{\substack{ 0 \neq \mathcal{A} \subset
  O_K \\ N(\mathcal{A}) \ll N(\varpi^k)^{1/2+\varepsilon}}} \frac{1}{N(\mathcal{A})^{1/2}} \ll y^{3/4+\varepsilon}.
\end{align*}

   For $(a, 6)=1$, we define a Hecke character $\chi^{(a)} \pmod {36a}$ such that for any ideal $(c)$ co-prime to $6$, with $c$ being the unique primary generator of $(c)$, $\chi^{(a)}((c))$ is defined as $\chi^{(a)}((c))=\leg {a}{c}$. One checks easily that $\chi^{(a)}$ is a Hecke character $\pmod {36a}$ of trivial infinite type. By an abuse of notation, we shall also write $\chi^{(a)}(c)$ for $\chi^{(a)}((c))$.  In particular, we have $\chi_c(a)=\chi^{(a)}(c)$. Note that any integral non-zero ideal $\mathcal{A}$ in $\mz[\omega]$ has a unique generator
$2^{r_1}(1-\omega)^{r_2}a$, with $r_1, r_2\in \intz, r_1,r_2 \geq 0 , a \in \intz[\omega]$,  $(a, 2)=1, a \equiv 1 \pmod 3$, it follows from this and our discussions above that
$\chi_{c}(\mathcal{A}) = \chi^{(a)}(c)$. Thus we have
\[  M = 2 \sum_{\substack{r_1,r_2 \geq 0 \\ (a,2)=1 \\ a \equiv 1 \bmod 3}} \frac{1}{2^{r_1}3^{r_2/2}N(a)^{1/2}}M(r,a), \]
where
\[ M(r,a)= \sum_{c \equiv 1 \bmod {36}}\chi^{(a)}(c)\Lambda(c)V \left( \frac{\pi  2^{2r_1+1}3^{r_2-1/2}N(a)}{N(c)^{1/2}} \right )\Phi\left( \frac{N(c)}{y} \right). \]

  We set
\begin{align*}
  \tilde{f}(s)=\int\limits^{\infty}_{0}V \left( \frac{\pi 2^{2r_1+1}3^{r_2-1/2}N(a)}{(xy)^{1/2}} \right) \Phi(x) x^{s-1} \dif x.
\end{align*}

Integration by parts and using \eqref{2.07} shows that $\tilde{f}(s)$ is a function satisfying the bound for all $\Re(s) > 0$, and $E>0$,
\begin{align}
\label{3.1}
  \tilde{f}(s) \ll (1+|s|)^{-E} \left( 1+\frac{ 2^{2r_1+1}3^{r_2-1/2}N(a)}{y^{1/2}} \right)^{-E}.
\end{align}

   We now apply Mellin inversion to see that
\begin{align}
\label{Mra}
\begin{split}
   M(r,a) =& \sum_{\substack{ c \equiv 1 \bmod {36}}} \chi^{(a)}(c)\Lambda(c)\frac 1{2\pi i}\int\limits_{(2)} \left( \frac{y}{N(c)} \right)^s \tilde{f}(s) \ \dif s \\
   =& \frac 1{2\pi i}\int\limits_{(2)}\tilde{f}(s) y^s \sum_{\substack{ c \equiv 1 \bmod {36}}}\frac {\chi^{(a)}(c)\Lambda(c)}{N(c)^s} \dif s = \frac {1}{\#h_{(36)}}\sum_{\psi \bmod {36}}\frac {1}{2\pi
   i}\int\limits\limits_{(2)}\tilde{f}(s) y^s \left (- \frac {L'(s, \psi\chi^{(a)})}{L(s, \psi\chi^{(a)})} \right )\dif s.
\end{split}
\end{align}
   where the last equality above follows from using the ray class characters to detect the
condition that $c \equiv 1 \bmod {36}$ with $\psi$ running over all ray class characters $\pmod {36}$. \newline

   We shift the contour of integration to $1/2+\varepsilon$ in the last integral of \eqref{Mra} to evaluate $M$.  By doing so,
we encounter a pole at $s = 1$ of $L'(s, \psi\chi^{(a)})/L(s, \psi\chi^{(a)})$ when $\psi\chi^{(a)}$ is principal, mindful of our assumption of GRH.  We set $M_0$ to be the contribution to $M$ of these residues, and $M_1$ to be the remainder. \newline

   We treat $M_1$ by bounding everything by absolute values and use \eqref{3.1} to get that for any $E>0$,
\begin{align}
\label{3.2}
   M_1 \ll y^{1/2+\varepsilon} \sum_{\psi \bmod {36}}  \sum_{\substack{r_1,r_2 \geq 0 \\ (a,2)=1 \\ a \equiv 1 \bmod 3}}
\frac{1}{2^{r_1}3^{r_2/2}N(a)^{1/2}}\left( 1+\frac{ 2^{2r_1+1}3^{r_2-1/2}N(a)}{y^{1/2}} \right)^{-E}\int\limits^{\infty}_{-\infty}
\left| \frac {L'\left(1/2+it, \psi\chi^{(a)} \right)}{L \left(1/2+it, \psi\chi^{(a)}\right)} \right| (1+|t|)^{-E} \dif t.
\end{align}
  Note that it follows from \cite[Theorem 5. 17]{iwakow} that
\begin{align*}
    \frac {L'\left(1/2+it, \psi\chi^{(a)} \right)}{L \left(1/2+it, \psi\chi^{(a)} \right)} \ll (\log (N(a)(1+|s|)))^{1+\varepsilon}.
\end{align*}

  Applying this in \eqref{3.2} and note that we can restrict the sum over $r_1$, $r_1$, $a$ to be $2^{2r_1+1}3^{r_2-1/2}N(a) \leq y^{1/2+\varepsilon}$, we immediately deduce that
\begin{align}
\label{M1}
   M_1 \ll y^{3/4+\varepsilon}.
\end{align}

    To determine $M_0$, we note that $\psi\chi^{(a)}$ is principal if and only if both $\psi$ and $\chi^{(a)}$ are principal. Hence $a$ must be a square. We denote $\psi_0$ for the principal
    ray class character $\pmod {36}$. Then we have
\begin{align*}
   L(s, \psi_0\chi^{(a^2)})=\zeta_{\mq(\omega)}(s)\prod_{(\varpi) |(6a)} \left(1-N(\varpi)^{-s} \right).
\end{align*}

   Since $\zeta_{\mq(\omega)}(s)$ has a simple pole at $s=1$, it follows that
\begin{align*}
  M_0 &=\frac {2y}{\#h_{(36)}} \sum_{\substack{r_1,r_2 \geq 0 \\ (a,2)=1 \\ a \equiv 1 \bmod 3}}  \frac{1}{2^{r_1}3^{r_2/2}N(a)}\tilde{f}(1)
  \\
  &=\frac {2y}{\#h_{(36)}} \sum_{\substack{r_1,r_2 \geq 0}}  \frac{1}{2^{r_1}3^{r_2/2}} \int\limits_{\mr}\Phi(x)
\sum_{\substack{(a,2)=1 \\ a \equiv 1 \bmod 3}} \frac{1}{N(a)}V \left( \frac{\pi 2^{2r_1+1}3^{r_2-1/2}N(a)^2}{(xy)^{1/2}} \right) \dif x.
\end{align*}

   Now, we apply the definition of $V$ given in \eqref{2.14} corresponding to $t=0$ there to see that
\begin{align}
\label{M0inner}
\begin{split}
& \sum_{\substack{(a,2)=1 \\ a \equiv 1 \bmod 3}} \frac{1}{N(a)}V \left( \frac{\pi 2^{2r_1+1}3^{r_2-1/2}N(a)^2}{(xy)^{1/2}} \right)\\
=& \frac {1}{2\pi
   i}\int\limits\limits_{(2)}\frac {\Gamma(s+1/2)}{\Gamma (1/2)} \Big ( \frac {(xy)^{1/2}}{\pi 2^{2r_1+1}3^{r_2-1/2}} \Big )^s \zeta_K(1+2s)
   \left ( \prod_{\mathfrak{p} | 6} \left( 1 - N(\mathfrak{p})^{-(1+2s)} \right) \right ) \ \frac
   {\dif s}{s} .
\end{split}
\end{align}
  Here and in what follows, we denote $\mathfrak{p}$ for prime ideals in $\mathcal{O}_K$. \newline

   By shifting the line of integration in \eqref{M0inner} to $\Re(s)=-1/4+\varepsilon$,  we encounter a double pole at $s=0$ with residue being  (by taking note that the residue of $\zeta_K(s)$ at $s = 1$ is $\sqrt{3}\pi/9$)
\begin{align*}
\displaystyle \frac {\pi}{12 \sqrt{3}}\log \frac {(xy)^{1/2}}{\pi 2^{2r_1+1}3^{r_2-1/2}} +B 
\end{align*}
  for some constant $B$.  Now, the convexity bound for $\zeta_K(s)$ (see \cite[Exercise 3, p. 100]{iwakow}) implies that for $\Re(s) =-1/4+\varepsilon$,
\begin{align*}
  \zeta_K(1+2s) \ll \left( 1+|s|^2 \right)^{1/4+\varepsilon}.
\end{align*}
  It follows from this that the integral over the line $\Re(s)=-1/4+\varepsilon$ gives a contribution of $O(y^{3/4+\varepsilon})$ to $M_0$.  We thus conclude by noting that $\#h_{(36)}=108$ and a straightforward calculation that
\begin{align} \label{M0}
  M_0 =\frac {(1+\sqrt{3})\pi}{1296}\hat{\Phi}(0) y\log y +B_{\mq(\omega)} y+O \left( y^{3/4+\varepsilon} \right),
\end{align}
   for some constant $B_{\mq(\omega)}$, which is the same constant $B_{\mq(\omega)}$ appearing in \eqref{1stmom}. \newline

  Combining \eqref{sumprimetointeger}, \eqref{M1} and \eqref{M0}, the proof of Theorem \ref{firstmoment} is now complete.

\section{Proof of Theorem~\ref{secmom}}
\label{section Thm2}
   Once again the proofs for $K=\mq(\omega)$ and $\mq(i)$ are similar.  So we shall fix $K=\mq(\omega)$ throughout this section to
give a proof for this case only.
 Starting with the approximate functional equation \eqref{approxfuneq},
and similar to the derivation of \eqref{sumprimetointeger}, we have
\begin{align}
\label{sumprimetointeger1}
\begin{split}
  & \sum_{\varpi \equiv 1 \bmod {9}}L \left( \frac{1}{2},
   \chi_{3,\varpi} \right) \Lambda(\varpi) \Phi\left( \frac{N(\varpi)}{y} \right)  \\
 =&  \sum_{\varpi  \equiv 1 \bmod {9}} \ \sum_{0 \neq \mathcal{A} \subset
  \mathcal{O}_K} \frac{\chi_{3,\varpi} (\mathcal{A})\Lambda(\varpi)}{N(\mathcal{A})^{1/2}}V \left(\frac{2\pi}{x}
  N(\mathcal{A} ) \right)\Phi \left( \frac{N(\varpi)}{y} \right)
 \\
& \hspace*{2cm} +\sum_{\varpi \equiv 1 \bmod {9}}\frac{W(\chi_{3,\varpi} )}{N(\varpi)^{1/2}} \sum_{0 \neq \mathcal{A} \subset
  \mathcal{O}_K}\frac{\overline{\chi}_{3, \varpi}(\mathcal{A})\Lambda(\varpi)}{N(\mathcal{A})^{1/2}}V \left( \frac{2\pi
  N(\mathcal{A})x}{|D_K|N(\varpi)} \right)\Phi \left( \frac{N(\varpi)}{y} \right) \\
=& \sum_{c \equiv 1 \bmod {9}} \ \sum_{0 \neq \mathcal{A} \subset
  \mathcal{O}_K} \frac{\chi_{3,c}(\mathcal{A})\Lambda(c)}{N(\mathcal{A})^{1/2}}V \left(\frac{2\pi}{x}
  N(\mathcal{A} ) \right)\Phi \left( \frac{N(c)}{y} \right) \\
&\hspace*{2cm}  + \sum_{\varpi \equiv 1 \bmod {9}}\frac{W(\chi_{3,\varpi})}{N(\varpi)^{1/2}} \sum_{0 \neq \mathcal{A} \subset
  \mathcal{O}_K}\frac{\overline{\chi}_{3,\varpi}(\mathcal{A})\Lambda(\varpi)}{N(\mathcal{A})^{1/2}}V \left( \frac{2\pi
  N(\mathcal{A})x}{|D_K|N(\varpi)} \right)\Phi \left( \frac{N(\varpi)}{y} \right)+O \left( y^{1/2+\varepsilon}x^{1/2+\varepsilon} \right) \\
:=& {\sum}_1 +  {\sum}_2+ O \left( y^{1/2+\varepsilon}x^{1/2+\varepsilon} \right),
\end{split}
\end{align}
   where $x>1$ is to be chosen later and the summation in ${\sum}_1$ is over all elements $c \in \mz[\omega]$. \newline

  Similar to our discussions in Section \ref{section Thm1}, we can write any integral non-zero ideal $\mathcal{A}$ in $\mz[\omega]$ as
$\mathcal{A}=((1-\omega)^{r}a)$, with $r \in \intz, r \geq 0 , a \in \intz[\omega]$,  $a \equiv 1 \pmod 3$. For $(a,3)=1$,
we define a Hecke character $\chi^{(a)}_3 \pmod {9a}$ such that for any ideal $(c)$ co-prime to $3$, with $c$ being the unique primary generator of $(c)$,
$\chi^{(a)}_3((c))$ is defined as $\chi^{(a)}_3((c))=\leg {a}{c}_3$. One checks easily that $\chi^{(a)}_3$ is a Hecke character $\pmod {9a}$ of trivial infinite
type. By an abuse of notation, we shall also write $\chi^{(a)}_3(c)$ for $\chi^{(a)}_3((c))$. It follows from this and our discussions above that
$\chi_{3,c}(\mathcal{A}) = \chi^{(a)}_3(c)$ for $c \equiv 1 \bmod {9}$. By further noting that $W(\chi_{3,\varpi})=g_3(\varpi)$ by \cite[(3.86)]{iwakow}, we thus obtain
\begin{align*}
{\sum}_1=&  \sum_{\substack{r \geq 0 \\ a \equiv 1 \bmod 3}} \frac{1}{3^{r/2}N(a)^{1/2}}V\left( \frac{2\pi 3^{r}N(a)}{x}  \right ) M_3(r,a), \\
    {\sum}_2 =& \sum_{\substack{r \geq 0 \\ a \equiv 1 \bmod 3}} \frac{1}{3^{r/2}N(a)^{1/2}} \sum_{\varpi \equiv 1 \bmod {9}} \frac {g_3(\varpi)
\overline{\chi}^{(a)}_3(\varpi) \Lambda(\varpi)}{N(\varpi)^{1/2}} V \left( \frac{2\pi
  3^{r}N(a) x}{|D_K|N(\varpi)} \right)\Phi \left( \frac{N(\varpi)}{y} \right),
\end{align*}
where
\[ M_3(r,a)= \sum_{c \equiv 1 \bmod {9}}\chi^{(a)}_3(c)\Lambda(c)  \Phi \left( \frac{N(c)}{y} \right). \]

  We treat ${\sum}_2$ by using the ray class characters to detect the
condition that $c \equiv 1 \bmod {9}$ and then applying Lemma \ref{lemg3} and partial summation to see that (note that we can assume $N(\varpi) \ll y^{1+\varepsilon}$ and $3^rN(a)x \ll  y^{1+\varepsilon}$ in view of the exponential decay of the test functions)
\begin{align*}
   \sum_{\varpi \equiv 1 \bmod {9}} & \frac {g_3(\varpi)
\overline{\chi}^{(a)}_3(\varpi) \Lambda(\varpi)}{N(\varpi)^{1/2}} V \left( \frac{2\pi
  3^{r}N(a) x}{|D_K|N(\varpi)} \right)\Phi \left( \frac{N(\varpi)}{y} \right) \\
=& \sum_{\psi \bmod {9}} \ \sum_{\varpi \equiv 1 \bmod {3}}\frac {\psi(\varpi)g_3(\varpi)
\overline{\chi}^{(a)}_3(\varpi) \Lambda(\varpi)}{N(\varpi)^{1/2}}  V \left( \frac{2\pi
  3^{r}N(a) x}{|D_K|N(\varpi)} \right)\Phi \left( \frac{N(\varpi)}{y} \right) \\
   \ll & \int\limits^{y^{1+\varepsilon}}_1 V \left( \frac{2\pi
  3^{r}N(a) x}{|D_K|u} \right)\Phi \left( \frac{u}{y} \right ) \dif O\left( u^{\varepsilon} \left( N(a)^{1/8}u^{27/32}+u^{1-1/20} \right) \right)
   \ll N(a)^{1/8}y^{27/32+\varepsilon}+y^{1-1/20+\varepsilon},
\end{align*}
    where $\psi$ runs over all ray class characters $\pmod {9}$. \newline

    It follows that
\begin{align}
\label{sum2}
    {\sum}_2 \ll \sum_{\substack{3^rN(a)x \ll  y^{1+\varepsilon}}} \frac{1}{3^{r/2}N(a)^{1/2}} \left( N(a)^{1/8}y^{27/32+\varepsilon}+y^{1-1/20+\varepsilon} \right)
 \ll  y^{27/32+\varepsilon} \left( \frac yx \right)^{1/2+1/8}+y^{1-1/20+\varepsilon} \left(\frac yx \right)^{1/2}.
\end{align}

  To treat $\sum_1$, we apply Mellin inversion to see that
\[   M_3(r,a) = \frac 1{2\pi i}\int\limits_{(2)}\widetilde{\Phi}(s) y^s\sum_{c \equiv 1 \bmod {16}}\frac {\chi^{(a)}_4(c)\Lambda(c)}{N(c)^s} \dif s,
\quad \mbox{where} \quad \widetilde{\Phi}(s)=\int\limits^{\infty}_{0}\Phi(x) x^{s-1} \dif x.\]

 We note here that integration by parts shows that $\widetilde{\Phi}(s)$ satisfies the following bound for all $\Re(s) > 0$, and $E_1>0$,
\begin{align*}
  \widetilde{\Phi}(s) \ll (1+|s|)^{-E_1} .
\end{align*}

   We now use the ray class characters again to detect the
condition that $c \equiv 1 \bmod {9}$, getting
\begin{align}
\label{M3}
\begin{split}
  M_3(r,a)=\frac {1}{\#h_{(9)}}\sum_{\psi \bmod {9}}\frac {1}{2\pi
   i}\int\limits\limits_{(2)}\widetilde{\Phi}(s) y^s \left (-\frac {L'(s, \psi\chi^{(a)}_3)}{L(s, \psi\chi^{(a)}_3)} \right ) \dif s,
\end{split}
\end{align}
   where $\psi$ runs over all ray class characters $\pmod {9}$, $\#h_{(9)}=9$. \newline

   We estimate $\sum_1$ by shifting the contour of integration in \eqref{M3} to the line $\Re(s)=1/2+\varepsilon$. We encounter poles at $s=1$ when
$\psi\chi^{(a)}_3$ is principal and we set $M_{0,3}$ to be the contribution to $\sum_1$ of these residues, and $M_{1,3}$ to be the
remainder. To treatment of $M_{1,3}$ can be done similar to that of $M_1$ in Section \ref{section Thm1}, and we have
\begin{align}
\label{M13}
   M_{1,3} \ll y^{1/2+\varepsilon}x^{1/2+\varepsilon}.
\end{align}

    We now determine $M_{0,3}$ by noting that $\psi\chi^{(a)}_3$ is principal if and only if both $\psi$ and $\chi^{(a)}_3$ are principal.
Hence $a$ must be a cubic. Similar to the treatment in Section \ref{section Thm1}, we denote $\psi_0$ for the principal
    ray class character $\pmod {9}$ to see that
\begin{align*}
  M_{0,3} &=\frac {y}{\#h_{(9)}}\hat{\Phi}(0) \sum_{\substack{r \geq 0 \\ a \equiv 1 \bmod {3}}} \frac{1}{3^{r/2}N(a)^{3/2}}V\left( \frac{2\pi 3^{r/2}N(a)^3}{x}  \right ).
\end{align*}

   We now apply \eqref{2.07} to see that
\begin{align*}
 \sum_{\substack{r \geq 0 \\ a \equiv 1 \bmod {3}}} & \frac{1}{3^{r/2}N(a)^{3/2}}  V\left( \frac{2\pi 3^{r}N(a)^3}{x}  \right ) \\
  =&
 \sum_{\substack{r \geq 0 \\ a \equiv 1 \bmod {3} \\ 2\pi 3^{r}N(a)^3 \leq x }} \frac{1}{3^{r/2}N(a)^{3/2}}
\left( 1+\left( \frac{2\pi 3^{r}N(a)^3}{x}  \right )^{1/2-\varepsilon} \right ) + \sum_{\substack{r \geq 0 \\ a \equiv 1 \bmod {3} \\ 2\pi 3^{r}N(a)^3 > x }} \frac{1}{3^{r/2}N(a)^{3/2}}\left( \frac {x}{2\pi 3^{r}N(a)^3}  \right ) \\
 =& \frac {3\sqrt{3}-1}{3\sqrt{3}-3} \zeta_{\mq(\omega)} \left(\frac 32 \right)+O \left( \frac {1}{x^{1/6}} \right).
\end{align*}

  We then conclude that by a straightforward calculation that
\begin{align} \label{M0formula}
  M_{0,3} =C_{\mq(\omega)}\hat{\Phi}(0) y +O(yx^{-1/6}),
\end{align}
  where $C_{\mq(\omega)}$ is given in \eqref{C}. \newline

  By combining the estimations given in \eqref{sumprimetointeger1}, \eqref{sum2}, \eqref{M13} and \eqref{M0formula}, we obtain that
\begin{align*}
   \sum_{\varpi \equiv 1 \bmod {9}}L \left( \frac{1}{2},
   \chi_{3,\varpi} \right) & \Lambda(\varpi) \Phi\left( \frac{N(\varpi)}{y} \right) \\
=& C_{\mq(\omega)} y +O \left( yx^{-1/6}+y^{1/2+\varepsilon}x^{1/2+\varepsilon}+y^{27/32+\varepsilon} \left( \frac yx \right)^{1/2+1/8}+y^{1-1/20+\varepsilon} \left( \frac yx \right)^{1/2} \right).
\end{align*}

   The assertion of Theorem \ref{secmom} now follows from this by setting $x=y^{19/20}$.

\section{Proof of Theorem \ref{thirdmom}}
\label{sec thirdthm}

  As the proof of Theorem \ref{thirdmom} is similar to that of Theorem \ref{secmom}, we shall only give a sketch of the proof for the case $j=3$.  We let
\begin{equation*}
\mathcal{M} = \sum_{\substack{ p=N(\varpi) \\ \varpi \equiv 1 \bmod {9}} }\;
\sum_{\substack{\chi \bmod{p} \\ \chi^3 = \chi_0, \chi \neq \chi_0}} L(1/2, \chi) \Lambda(p) \Phi \leg{p}{Q}.
\end{equation*}

   Applying the approximate functional equation \eqref{approxfunc} with $A_p B = p$, we obtain $\mathcal{M} = \mathcal{M}_1 + \mathcal{M}_2$, where
\begin{align*}
 \mathcal{M}_1 &= \sum_{\substack{ p=N(\varpi) \\ \varpi \equiv 1 \bmod {9}} }\;  \sum_{\substack{\chi \bmod{p} \\ \chi^3 = \chi_0, \chi \neq \chi_0}}
\sum_{m=1}^{\infty} \frac{\chi(m)\Lambda(p)}{\sqrt{m}} W\leg{m}{x} \Phi\left(\frac{p}{Q}\right), \\
 \mathcal{M}_2 &= \sum_{\substack{ p=N(\varpi) \\ \varpi \equiv 1 \bmod {9}} }\;  \sum_{\substack{\chi \bmod{p} \\ \chi^3 = \chi_0, \chi \neq \chi_0}}
\frac {\tau(\chi)}{q^{1/2}}\sum_{m=1}^{\infty} \frac{\overline{\chi}(m)\Lambda(p)}{\sqrt{m}} W\leg{x m}{p} \Phi\left(\frac{p}{Q}\right),
\end{align*}
   for some $x>1$ to be specified later. \newline

   It follows from Lemma \ref{lemma:quarticclass} and \eqref{taucubic} that we have
\begin{align*}
 \mathcal{M}_1 &= \sum_{\substack {\varpi \equiv 1 \bmod 9}} \ \sum_{m=1}^{\infty} \frac{\chi_{3,\varpi}(m)\Lambda(\varpi)}{\sqrt{m}} W \leg{m}{x}
\Phi\left(\frac{N(\varpi)}{Q}\right), \\
 \mathcal{M}_2 &= \sum_{\substack {\varpi \equiv 1 \bmod 9}} \frac {g_3(\varpi)}{N(\varpi)^{1/2}} \sum_{m=1}^{\infty} \frac{\overline{\chi}_{3, \varpi}(m)\Lambda(\varpi)}{\sqrt{m}}
W\leg{x m}{N(\varpi)} \Phi \left(\frac{N(\varpi)}{Q}\right).
\end{align*}

   We argue as in Section \ref{section Thm1} to see that
\begin{align}
\label{calM1}
 \mathcal{M}_1 &= \sum_{\substack {c \equiv 1 \bmod 9}} \ \sum_{m=1}^{\infty} \frac{\chi_{3,c}(m)\Lambda(c)}{\sqrt{m}} W \leg{m}{x}
\Phi\left(\frac{N(c)}{Q}\right)+O\left( Q^{1/2+\varepsilon}x^{1/2+\varepsilon} \right) =: \mathcal{M}'_1+O \left( Q^{1/2+\varepsilon}x^{1/2+\varepsilon} \right).
\end{align}

   Let the Hecke character $\chi^{(a)}_3 \pmod {9a}$ be defined as in Section \ref{section Thm1} for $(a,3)=1, a \in \mz[\omega]$.
Similar to our discussions in Section \ref{section Thm1}, on writing every positive $m \in \mz$ as $m=3^{r}m'$, with $r \in \intz, r \geq 0,
m' \in \intz, (m', 3)=1$, we see that $\chi_{3,c}(m) = \chi^{(m')}_3(c)$. It follows  that
\begin{align*}
 \mathcal{M}'_{1} =&  \sum_{\substack{r \geq 0 \\ (m,3)=1, m > 0}} \frac{1}{3^{r/2}\sqrt{m}} W \leg{3^r m}{x} \mathcal{M}_1(r,m),  \\
  \mathcal{M}_2 =& \sum_{\substack{r \geq 0 \\ (m,3)=1, m > 0}} \frac{1}{3^{r/2}\sqrt{m}} \sum_{\substack {\varpi \equiv 1 \bmod 9}}
\frac {g_3(\varpi)\overline{\chi}^{(m)}_{3}(\varpi) \Lambda(\varpi)}{N(\varpi)^{1/2}}
W\leg{3^r m x}{N(\varpi)} \Phi \left(\frac{N(\varpi)}{Q}\right),
\end{align*}
where
\[ \mathcal{M}_1(r,m)= \sum_{c \equiv 1 \bmod {9}}\chi^{(m)}_3(c)\Lambda(c)  \Phi \left( \frac{N(c)}{Q} \right). \]

   We apply Mellin inversion and shift the contour of integration to find that
\begin{align}
\label{M1'}
    \mathcal{M}'_{1} =D_3 Q +O \left( Qx^{-1/6}+Q^{1/2+\varepsilon}x^{1/2+\varepsilon} \right),
\end{align}
   where $D_3$ is given as in \eqref{D}. \newline

  We treat $\mathcal{M}_{2}$ by using the ray class characters to detect the
condition that $c \equiv 1 \bmod {9}$ and then applying Lemma \ref{lemg3} and partial summation to see that
\begin{align*}
   &   \sum_{\substack {\varpi \equiv 1 \bmod 9}}
\frac {g_3(\varpi)\overline{\chi}^{(m)}_{3}(\varpi) \Lambda(\varpi)}{N(\varpi)^{1/2}}
W\leg{3^r m x}{N(\varpi)} \Phi \left(\frac{N(\varpi)}{Q}\right)
   \ll  N(m)^{1/8}Q^{27/32+\varepsilon}+Q^{1-1/20+\varepsilon}.
\end{align*}

   We then deduce from this via partial summation that
\begin{align}
\label{M2}
\begin{split}
    \mathcal{M}_{2}  \ll   Q^{27/32+\varepsilon} \left( \frac Qx \right)^{1/2+1/4}+Q^{1-1/20+\varepsilon} \left( \frac Qx \right)^{1/2}.
\end{split}
\end{align}

   By combining the estimations given in \eqref{calM1}, \eqref{M1'} and \eqref{M2}, we obtain that
\begin{align*}
  \sum_{\substack{ p=N(\varpi) \\ \varpi \equiv 1 \bmod {9}} }\;
\sum_{\substack{\chi \bmod{p} \\ \chi^3 = \chi_0, \chi \neq \chi_0}} & L(1/2, \chi) \Lambda(p) \Phi \leg{p}{Q} \\
=& D_3 Q +O \left( Qx^{-1/6}+Q^{1/2+\varepsilon}x^{1/2+\varepsilon}+Q^{27/32+\varepsilon} \left( \frac Qx \right)^{1/2+1/4}+Q^{1-1/20+\varepsilon} \left( \frac Qx \right)^{1/2} \right).
\end{align*}
   The assertion of Theorem \ref{thirdmom} for $i=3$ now follows from this by setting $x=Q^{19/20}$.

\vspace*{.5cm}

\noindent{\bf Acknowledgments.} P. G. is supported in part by NSFC grant 11871082 and L. Z. by the FRG grant PS43707 at UNSW.

\bibliography{biblio}
\bibliographystyle{amsxport}

\vspace*{.5cm}

\noindent\begin{tabular}{p{8cm}p{8cm}}
School of Mathematical Sciences & School of Mathematics and Statistics \\
Beihang University & University of New South Wales \\
Beijing 100191 China & Sydney NSW 2052 Australia \\
Email: {\tt penggao@buaa.edu.cn} & Email: {\tt l.zhao@unsw.edu.au} \\
\end{tabular}

\end{document}